\documentclass[12pt]{article}

\usepackage{amsmath}
\usepackage{amsfonts}
\usepackage{latexsym}
\usepackage{amssymb}
\usepackage{enumerate}
\usepackage{amsthm}
\usepackage[normalem]{ulem}
\usepackage{csquotes}
\usepackage{lineno,xcolor}
\usepackage[normalem]{ulem}
\usepackage{filecontents}
\usepackage{scalefnt}
\usepackage{tikz}
\usepackage{fancyhdr}
\usetikzlibrary{arrows}
\usetikzlibrary{positioning}

\usepackage{lineno,xcolor}
\definecolor{amaranth}{rgb}{0.9, 0.17, 0.31}
\definecolor{bluegray}{rgb}{0.4, 0.6, 0.8}

\makeatletter

\newtheorem*{maintheorem*}{Main Theorem}
\newtheorem{theorem}{Theorem}[section]

\newtheorem*{theorem*}{Theorem}

\newtheorem{remark}[theorem]{Remark}

\newtheorem*{example*}{Example}
\newtheorem*{conjecture*}{Conjecture}
\def\1{\mathbf 1}

\def\u{\mathbf u}

\def\0{\mathbf 0}

\def\cB{\mathcal B}
\def\cC{\mathcal C}
\def\cD{\mathcal D}

\def\cF{\mathcal F}

\def\cH{\mathcal H}

\def\cJ{\mathcal J}

\def\cK{\mathcal K}
\def\cO{\mathcal O}

\def\PG{{\rm PG}}
\def\PGL{{\rm PGL}}
\def\PGammaL{{\rm P\Gamma L}}

\def\PGammaO{{\rm P\Gamma O}}

\def\GF{{\rm GF}}
\def\GL{{\rm GL}}

\def\Aut{{\rm Aut}}

\def\GF{{\rm GF}}

\def\mod{{\rm mod} }

\def\Tr{{\rm Tr}}

\def\<{\langle}
\def\>{\rangle}
\newcommand\comment[1]{}

\newcommand*{\shifttext}[2]{
  \settowidth{\@tempdima}{#2}
  \makebox[\@tempdima]{\hspace*{#1}#2}
}

\newcommand\redsout{\bgroup\markoverwith{\textcolor{amaranth}{\rule[0.5ex]{2pt}{0.4pt}}}\ULon}

\newcommand\redout{\bgroup\markoverwith
{\textcolor{red}{\rule[.4ex]{2pt}{0.8pt}}}\ULon}

\title{Classification of flocks \\ of the quadratic cone in $\PG(3,64)$}
\author{Giusy Monzillo\footnote{The research was supported by the Italian National Group for Algebraic and Geometric Structures and their Applications (GNSAGA-INdAM). } \\
\small {\tt giusy.monzillo@unibas.it}\\[0.8ex]
\small Dipartimento di Matematica, Informatica ed Economia\\[-0.8ex]
\small Universit\`a degli Studi della Basilicata\\[-0.8ex]
\small Viale dell'Ateneo Lucano 10 \\[-0.8ex]
\small 85100 Potenza, Italy\\
\and   
Tim Penttila \\ 
\small {\tt penttila86@msn.com}\\
\small School of Mathematical Sciences\\[-0.8ex]
\small The University of Adelaide\\[-0.8ex]
\small Adelaide, South Australia \\[-0.8ex]
\small 5005 Australia
\and
Alessandro Siciliano$^*$ \\
\small{\tt alessandro.siciliano@unibas.it}\\[0.8ex]
\small Dipartimento di Matematica, Informatica ed Economia\\[-0.8ex]
\small Universit\`a degli Studi della Basilicata\\[-0.8ex]
\small Viale dell'Ateneo Lucano 10 \\[-0.8ex]
\small 85100 Potenza, Italy\\
}
\date{}

\begin{document}


\maketitle

\begin{abstract}
Flocks are an important topic in the field of finite geometry, with many relations with other objects of interest. This paper is a contribution to the difficult problem of classifying flocks up to projective equivalence.
We complete the classification of flocks of the quadratic cone in $\PG(3,q)$ for $q \le 71$, by showing by computer that there are exactly three flocks of the quadratic cone in $\PG(3,64)$, up to equivalence. The three flocks had previously been discovered, and they are the linear flock, the Subiaco flock and the Adelaide flock. The classification proceeds via the connection between flocks and herds of ovals in $\PG(2,q)$, $q$ even, and uses the prior classification of hyperovals in $\PG(2,64)$.
\end{abstract}



\section{Introduction}

Circle planes were introduced by van der Waerden and Smid in \cite{vws} as an incidence theoretic analogue of the study of quadrics in projective 3-space containing a non-empty non-degenerate conic.
They come in three kinds: inversive planes (corresponding to elliptic quadrics), Minkowski planes (corresponding to hyperbolic quadrics) and Laguerre planes (corresponding to quadratic cones).
There are non-classical circle planes, such as inversive planes arising from ovoids that are not elliptic quadrics, Minkowski planes arising from sharply 3-transitive sets of permutations other than  $\PGL(2,F)$ and Laguerre planes arising from cones with base an oval that is not a conic. Van der Waerden and Smid \cite{vws} (for inversive and Laguerre planes) and Kaerlein \cite{kaer} (for Minkowski planes) characterized the circle planes arising from quadrics as those that satisfy  Miquel's Theorem \cite{dela}. The study of flocks of the quadratic cone grew out of an analogous study of flocks of Miquelian inversive and Minkowski planes in the period 1968\,-1976. The foundational papers on flocks of the quadratic cone are \cite{walker} and \cite{ft}.

Flocks of finite Miquelian circle planes were classified for inversive planes in even characteristic by  Thas \cite{thas2} and  in odd characteristic by Orr \cite{orr} (correcting an error of Dembowski \cite{dem});   for Minkowski planes the classification was obtained by combining the results of  Thas \cite{thas3}, Bader and  Lunardon \cite{bl}  and  Thas \cite{thas4}, with a simplified proof in \cite{ds}. There is little hope for a corresponding result for flocks of Miquelian Laguerre planes:  rather the best that can be done is
classification for small order. This paper completes the classification of those of order  up to 71 by supplying the last omitted case (that of order 64). A brief overview of the subject is given in \cite{pen2}.

Flocks play an important role in finite geometry. Here, we summarize some of the main lines of research involving flocks.
One of the original motivations for studying flocks of the quadratic cone was the Thas-Walker construction of a translation plane (via the connection of translation planes with spreads from  the construction of Andr\'e \cite{andr} and Bruck-Bose \cite{bb}) which appears in Walker \cite{walker} and Fisher-Thas \cite{ft}. (A similar
construction of translation planes from flocks of Miquelian Minkowski planes was given by Thas \cite{thas3} and Walker \cite{walker}; and the corresponding construction of translation planes from flocks of Miquelian inversive planes was given by Thas \cite{thas2} and used by Orr \cite{orr}.)

A method of constructing generalized quadrangles was developed  by Kantor and Payne from 1980 to 1986 \cite{kan80, pay80, pay85, kan86} which was linked to the study of flocks of the quadratic cone by Thas in 1987 \cite{thas1} giving a further motivation for their study. This paper of Thas inspired the paper 	\cite{gj} where the conditions on a translation plane to arise from a flock of the quadratic cone via the Thas-Walker construction were established. Furthermore, the paper \cite{gj} contains an easy proof that isomorphic flocks correspond to isomorphic translation planes, and conversely; this means that the correspondence between flocks and translation planes via the Thas-Walker construction is functorial.

In \cite{dst} sets of points in a generalized quadrangle meeting every line in 0 or 2 points were introduced; in \cite{dst2} these sets in the classical quadrangle $Q(4,q)$, $q$ odd, were used to characterize the quadrangle itself. Later, in \cite{blt}, the same sets in $Q(4,q)$, $q$ odd, were used to construct new flocks of the
quadratic cone derived from old ones, and Kantor named them {\em BLT-sets} in \cite{kan91}, after the intial letters of the surnames of the authors of  \cite{blt}. The corresponding procedure in even characteristic was introduced by Payne and Thas in \cite{pt2}. In  \cite{knarr}, Knarr simplified the connection between BLT-sets and generalized quadrangles.  The Knarr construction clarified (in odd characteristic) the non-functorial nature of the connection between flocks of the quadratic cone and generalized quadrangles, as the connection between BLT-sets and generalized quadrangles is functorial (isomorphic BLT-sets correspond to isomorphic generalized quadrangles)  while that between  BLT-sets and flocks of the quadratic cone is not (isomorphic flocks correspond to isomorphic BLT-sets with a distinguished point). Thus, isomorphic flocks correspond to isomorphic generalized quadrangles with a distinguished line
on the elation point.

A connection between flocks of the quadratic cone in even characteristic and ovals was developed by Payne  \cite{pay88}  and Cherowitzo {\em et al.}  \cite{cppr} which led to a collection of ovals that,  in the latter paper, was named a {\em herd}.
The action of $\PGammaL(2,q)$ introduced in \cite{okp} shows that herds of ovals play a role in even characteristic analogous to that of BLT-sets in odd characteristic: the correspondence between herds of ovals and generalized quadrangles is functorial, while that between herds of ovals and flocks of the quadratic cone is not; again, isomorphic flocks correspond to isomorphic generalized quadrangles with a distinguished line on the elation point.

A hyperbolic fibration is a partition of the projective 3-space into two lines and a certain number of hyperbolic quadrics. The study of hyperbolic fibrations goes back to  Denniston \cite{denn} and was developed, principally, by Baker and Ebert in a series of papers from 1988 to 2005  \cite{be88,ebert89,jpw91,bdew99,bew01} before being brought under the umbrella of
a single construction from flocks of the quadratic cone  by Baker, Ebert and Penttila in  \cite{bep05}. This gives a second connection between flocks of the quadratic cone and translation planes.

Further, subsequent, connections with flocks of the quadratic cone also exist - with partial geometries (mediated by hemisystems) and with spreads of generalized quadrangle (mediated by ovals).
The quadrangles arising from flocks contain hemisystems, from which partial geometries arise \cite{bgr}.  The ovals of a herd give rise to Tits quadrangles, which, in turn, admit spreads \cite{bokppr}.

Let us now turn to the important problem of classifying flocks up to projective equivalence. The present paper is a contribution to this problem.
All flocks of the quadratic cone of $\PG(3,q)$ have been classified for $q$ up to 71, except for $q=64$, by using in most of cases the equivalences discussed above. We emphasise that these kinds of problems are best attacked via the objects with the coarsest equivalence relations. It is much easier
to determine the flocks after having determined the BLT-sets or the herds, and much easier to determine the hyperbolic fibrations after determining the flocks, as each simply requires an orbit calculation of a group already determined. Moreover, the lists involved are shorter this way - there
is less hair-splitting.

In this paper,  we  classify the herds in $\PG(2,64)$, up to isomorphism. By the correspondence  between flocks of the quadratic cone  in $\PG(3,q)$ and herds of ovals in $\PG(2,q)$ for $q$ even, this is equivalent to classifying the flocks of the quadratic cone in $\PG(3,64)$, up to equivalence. That we are able to do this depends upon Vandendriessche's recent determination of the hyperovals of $\PG(2,64)$ \cite{vander}.
This is a second reason to approach this problem via herds: applying the knowledge of hyperovals
is easiest in the herd setting.

\section{Preliminaries and known results}\label{sec_2}

Let	$\cK$ be a quadratic cone in $\PG(3,q)$	with vertex $V$. A {\em flock} of $\cK$ is a set of $q$ planes partitioning the cone minus  $V$ into disjoint conics. 
If all the planes of the flock meet in an (external) line, the flock  is called {\em linear}, and linear flocks exist for all $q$.  As all quadratic cones of $\PG(3,q)$ are equivalent under the action of $\PGammaL(4,q)$, we can choose $\cK$ to have equation $y^2-xz=0$, with the vertex $V(0,0,0,1)$. In accordance with this choice of coordinates, a flock of $\cK$ is a set of
planes $\cJ=\{[a_t,b_t,c_t,1]:t\in\GF(q)\}$. Two flocks of $\cK$ are {\em projectively equivalent} if there is a collineation of $\PG(3,q)$ fixing $\cK$ and taking the first flock to the second. 

 On the set of all $2\times 2$ matrices over $\GF(q)$, there exists an equivalence relation defined as follows. Let  $A=\begin{pmatrix}x & y \\ w & z  \end{pmatrix}$ and $B=\begin{pmatrix}x' & y' \\ w' & z'  \end{pmatrix}$. Then $A\equiv B$ if and only if $x=x'$, $z=z'$ and $y+w=y'+w'$. Note that $A\equiv B$ if and only if $\alpha A\alpha^T=\alpha B\alpha^T$ if and only if $\alpha(A+A^T)=\alpha(B+B^T)$, for all $\alpha\in\GF(q)^2$.

A $q$-{\em clan} is a set of $q$ equivalence classes 
\[
\cC=\left\{\overline  A_t:t\in\GF(q)\right\},
\]
such that $A_s-A_t$ is anisotropic, that is $\u(A_s-A_t)\u^T=0$ if and only if $\u=0$. Note that the definition does not depend on the choice of the representatives of the classes in  $\cC$. Therefore, one may assume that $A_t=\begin{pmatrix}a_t & b_t \\ 0 & c_t  \end{pmatrix}$. In \cite{thas1} Thas proved that the set $\cC$ is a $q$-clan if and only if $\cJ(\cC)=\{[a_t,b_t,c_t,1]:t\in\GF(q)\}$ is a flock of $\cK$.

Two $q$-clans $\cC=\{\overline{A_t}:t\in\GF(q)\}$ and $\cC'=\{\overline{A'_t}:t\in\GF(q)\}$ are {\em equivalent} if there are $\lambda\in\GF(q)\setminus\{0\}$, $B\in\GL(2,q)$, $M$ a $2\times 2$ matrix over $\GF(q)$, $\sigma\in\Aut(\GF(q))$ and a permutation $\pi:t\mapsto  t'$ of $\GF(q)$ such that  
\[
A_{t'}\equiv\lambda BA_t^\sigma B^T+M,
\]
for all $t\in\GF(q)$. Hence, one can assume that $A_0$ is the zero matrix.  

Let $\cC=\{\overline{A_t}:t\in\GF(q)\}$ be a $q$-clan. Consider the group $G=\{(\alpha, c, \beta ) : \alpha, \beta \in \GF(q)^2, c\in \GF(q)\}$ with the binary operation $(\alpha, c, \beta)\cdot (\alpha', c', \beta')=(\alpha+\alpha', c+c'+\beta(\alpha')^T, \beta+\beta')$. Then, the following subgroups of $G$ are defined:
\[
A(\infty)=\{(\underline 0,0, \beta):\beta \in \GF(q)^2\}, \ \ \ A^*(\infty)=\{(\underline 0, c, \beta):c\in\GF(q),\beta \in \GF(q)^2\}
\]
\[
\begin{aligned}
A(t)&=\{(\alpha, \alpha A_t\alpha^T, \alpha (A_t+A_t^T)):\alpha \in \GF(q)^2\},\\[.1in]
 A^*(t)&=\{(\alpha, c, \alpha (A_t+A_t^T)):\alpha \in \GF(q)^2,c\in\GF(q)\},
 \end{aligned}
\]
for $t\in \GF(q)$. These form a so called 4\emph{-gonal family for G} \cite{kan86,pay85}, starting from which a generalized quadrangle GQ$(\cC)$ of order $(q^2,q)$ may be constructed in the standard way \cite{kan80}, \cite[A.1.1, A.1.2]{pt}.

The points are:

\begin{itemize}
\item[(i)]  the elements of $G$;
\item[(ii)] the cosets $A^*(t)g$, $t\in \GF(q) \cup \{\infty\}$, $g\in G$;
\item[(iii)]  the symbol $(\infty)$.
\end{itemize}

The lines are: 

\begin{itemize}
\item[(a)] the cosets $A(t)g$, $t\in \GF(q) \cup \{\infty\}$, $g\in G$;
\item[(b)] the symbols $[A(t)]$, $t\in \GF(q) \cup \{\infty\}$.
\end{itemize}

Incidence is defined by the following rules: the point $(\infty)$ is on the $q+1$ lines of type (b); a point  $A^*(t)g$ is on the line $[A(t)]$ and on the $q$ lines of type (a) contained in $A^*(t)g$; a point $g$ is on the $q+1$ lines $A(t)g$ that contain it. There are no other incidences. 

Therefore, a $q$-clan $\cC$ gives rise to the generalized quadrangle $\mathrm{GQ}(\cC)$. In particular $\mathrm{GQ}(\cC)$ is an {\em elation} GQ {\em about the (elation) point} $(\infty)$ \cite{kan86,pay80,pay85}. Because of the connection between $q$-clans and flocks of the quadratic cone, a $\mathrm{GQ}(\cC)$ is also called a {\em flock quadrangle}. 
\begin{theorem}\cite{pay96} {\em(The Fundamental theorem of $q$-clan geometry)}
For any prime power $q$, let $\cC = \{\overline{A_t}	: t \in\GF(q)\}$ and
$\cC' = \{\overline{A_t'}	: t \in\GF(q)\}$ be two (not necessarily distinct) $q$-clans 
such that $A_0= A_0'=\begin{pmatrix}0 & 0 \\ 0 & 0\end{pmatrix}$. Then, the following are equivalent: 
\begin{itemize}
 \item[(i)] $\cC$ and $\cC'$ are equivalent.
\item[(ii)] The flocks $\cJ(\cC)$ and $\cJ(\cC')$ are projectively equivalent.
\item[(iii)] $\mathrm{GQ}(\cC)$ and $\mathrm{GQ}(\cC')$ are isomorphic by an isomorphism mapping $(\infty)$ to $(\infty)$, $[A(\infty)]$ to $[A'(\infty)]$, and $(\underline 0, 0, \underline0)$ to $(\underline0, 0, \underline0)$.
\end{itemize}
\end{theorem}
\begin{remark}{\em 
In \cite{blp}, it is proved that, for a given $q$-clan $\cC$, there are, up to equivalence, $q+1$ flocks of the quadratic cone associated with $\mathrm{GQ}(\cC)$, one for each line through $(\infty)$. Two such flocks are projectively equivalent if and only if the corresponding lines belong to the same orbit of the stabilizer of $(\infty)$ in the automorphism group of the quadrangle. 
}
\end{remark}

When $q$ is odd, flocks of the quadratic cone are related to {\em BLT-sets} of the parabolic quadric $Q(4,q)$ in $\PG(4,q)$. A {\em BLT-set} of $Q(4,q)$ is a set $\cB$ of $q+1$ points such that every point of $Q(4,q)$ not in $\cB$ is collinear (on $Q(4,q)$) with 0 or 2 points of $\cB$. BLT-sets exist only for $q$ odd \cite{dst}.  Two  BLT-sets are equivalent if they are in the same orbit of the automorphism group $\PGammaO(5,q)$ of $Q(4,q)$.
\begin{theorem}\cite{blt}
Every flock of the quadratic cone in $\PG(3,q)$, $q$ odd, determines a BLT-set of $Q(4, q)$. Conversely, given a BLT-set $\cB$ of $Q(4,q)$, $q$ odd, and a point $P$ of $\cB$, there arises a flock of the quadratic cone of $\PG(3,q)$.	Moreover, equivalent flocks give equivalent BLT-sets, and conversely, two flocks arising from the points $P$ and $Q$ of a BLT-set $\cB$ are equivalent if and only if $P$ and $Q$ lie in the same orbit of the stabilizer of $\cB$ in $\PGammaO(5,q)$.
\end{theorem}

When $q$ is even,  flocks of the quadratic cone are related to {\em herds of ovals} in $\PG(2,q)$. 
An {\em oval} in $\PG(2,q)$ is a set of $q+1$ points, no three collinear. A {\em hyperoval} is a set of $q+2$ points, no three collinear, and hyperovals exist only when $q$ is even. 
Since $\PGL(3,q)$ is transitive on the ordered quadrangles of $\PG(2,q)$, we may assume that a hyperoval contains the fundamental quadrangle $\{(0, 0, 1),(0, 1,0), (1,0,0), (1, 1, 1)\}$. Thus, up to equivalence, a hyperoval $\cH$ can be represented as
\[
\cH = \{(1,t,f(t)):t\in\GF(q)\} \cup \{(0,0,1),(0,1,0)\},
\]
where $f$ is a permutation of $\GF(q)$, with $f(0) = 0$ and $f(1) = 1$.
Permutations that describe hyperovals in this way are called {\em o-polynomials} \cite{hir}.

Starting from an oval $\cO$ in $\PG(2,q)$ embedded in $\PG(3,q)$, Tits constructed a class of (non-classical) GQs of order $(q,q)$  \cite{dem,pt}. \\ The points are:  (i) the  points of $\PG(3,q)$ not in $\PG(2,q)$, (ii) the  planes of $\PG(3,q)$ meeting $\PG(2,q)$ in a unique point of $\cO$ and (iii) the symbol $(\infty)$.
\\
The lines are: (a) the lines of $\PG(3,q)$ not in $\PG(2,q)$ meeting $\PG(2,q)$ in a  point of $\cO$ and (b) the   points of $\cO$. 
\\ 
Incidences are as follows: a point of type (i) is incident only with the lines of type (a) which contain it,  a point of type (ii) is incident with all lines of type (a) contained in it and  with the unique line of type (b) on it,   the point  of type (iii) is incident with   all lines of type (b) and no line of type (a).
This GQ is denoted by $T_2(\cO)$.

\begin{theorem}\cite{pay85}
Let $q$ be even, $\cC$  a $q$-clan, $Q$  a point of GQ$(\cC)$ not collinear with 
$(\infty)$. Then, GQ$(\cC)$ has $q+1$ subquadrangles on $(\infty)$ and $Q$, each  isomorphic to $T_2(\cO_i)$, for some ovals $\cO_1,\ldots,\cO_{q+1}$ in $\PG(2,q)$.
\end{theorem}

The above theorem  led Cherowitzo, Penttila,  Pinneri and Royle \cite{cppr} to introduce the concept of a  herd of ovals, and to prove that in even characteristic every  flock of  the quadratic cone, via the associated $q$-clan, gives a herd of ovals, and conversely. 

	A {\em herd} of ovals in $\PG(2,q)$, $q$ even, is a family of $q+1$ ovals $\{\cO_s:s\in\GF(q)\cup\{\infty\}\}$, each of which  has nucleus $(0,0,1)$, contains the points $(1,0,0)$, $(0,1,0)$, $(1,1,1)$, and such that
\begin{align*}
 \cO_\infty  =  \{(1,t,f_\infty(t)):t\in\GF(q)\}\cup \{(0,1,0)\}\\[0.1in]
 \cO_s    =  \{(1,t,f_s(t)):t\in\GF(q)\}\cup \{(0,1,0)\}, & \,\,\,\, s\in\GF(q),
\end{align*}
where
\begin{equation}\label{eq_7}
f_s(t)=\frac{f_0(t)+\kappa sf_\infty(t)+s^{1/2} t^{1/2}}{1+\kappa s+s^{1/2}},\ \ \ s\neq 0,
\end{equation}
for some $\kappa\in\GF(q)$ with $\Tr(\kappa)=1$; here and below, $\Tr$ stands for the absolute trace function of $\GF(q)$, and the exponent ${1/2}$ represents the field automorphism $x\mapsto x^{q/2}$. Furthermore, $\cD(f_s)$ and $\cD(f_\infty)$ will be used to refer to $\cO_s$ and $\cO_\infty$ in the herd, respectively.

When $q$ is even, for a $q$-clan  $\cC=\left\{\overline  A_t:t\in\GF(q)\right\}$, 
 it is possible to assume that $A_t=\begin{pmatrix}f_0(t) & t^{1/2} \\ 0 & \kappa f_\infty(t)  \end{pmatrix}$, where $\Tr(\kappa)=1$, $f_0$ and $f_\infty$ are permutations of $\GF(q)$ satisfying $f_0(0)=0=f_\infty(0)$, $f_0(1)=1=f_\infty(1)$ and
\[
\Tr\left(\frac{\kappa(f_0(s)+f_0(t))(f_\infty(s)+f_\infty(t))}{s+t}\right)=1,
\] 
for all $s,t\in\GF(q)$ with $s\neq t$. A $q$-clan ($q$ even) written in this form is said to be {\em normalized}.

In \cite{cppr} it is proved that a herd of ovals gives rise to a (normalized) $q$-clan $\cC=\{\overline{A_t}:t\in\GF(q)\}$, with
\[
A_t= \begin{pmatrix}
f_0(t) & t^{{1/2}}\\
0 & \kappa f_\infty(t)
\end{pmatrix}.
\] 

Conversely, such a (normalized) $q$-clan $\cC$ is shown to correspond to a  herd  of ovals $\cH(\cC)$. We emphasize that this notation encloses all information, namely the o-polynomials $f_0$, $f_\infty$ and the parameter $\kappa$, to write the ovals in the herd.

In \cite{okp}  the so-called {\em magic action} was introduced. Here, we recall the definition and some results linked to it.   Let $\cF$ be the vector space of all functions $f:\GF(q)\rightarrow \GF(q)$ such that $f(0)=0$. It is well known that each element of $\cF$ can be written as a polynomial in one variable of degree at most $q-1$. Furthermore, if $f(x)=\sum{a_ix^i}$ and $\gamma\in\Aut(\GF(q))$, then  $f^\gamma(x)=\sum{a_i^\gamma x^i}$.

Any element $\psi=(A,\gamma)\in\PGammaL(2,q)$, with $A=\begin{pmatrix}a& b \\ c & d\end{pmatrix}$ and $\gamma\in\Aut(\GF(q))$, acts on $\cF$ by mapping $f$ to $\psi f$, where 
\[
\psi f(x)=|A|^{-1/2}\left[(bx+d)f^\gamma\left(\frac{ax+c}{bx+d}\right)+b\,x\,f^\gamma\left(\frac{a}{b}\right)+d\,f^\gamma\left(\frac{c}{d}\right)\right].
\]

This action of $\PGammaL(2,q)$ on $\cF$ is called  {\em magic action}. 
 
It is known that for each o-polynomial $f\in\cF$  we have that its degree is at most $q-2$, $f(1)=1$, and  the function $f_s$, where $f_s(0)=0$ and $f_s(x)=(f(x+s)+f(s))/x$, $x\neq 0$, is a permutation of $\GF(q)$, for each $s\in\GF(q)$. Conversely, for every polynomial $f\in\cF$  satisfying the above properties, the point-set $\cD=\{(1,t,f(t)):t\in\GF(q)\}\cup \{(0,1,0)\}$ is an oval of $\PG(2,q)$ with nucleus $(0,0,1)$ \cite{hir}. 
If  $f(1)=1$ is not required but the other conditions are, then $f$ is an {\em o-permutation over} $\GF(q)$. 

For an o-polynomial there are $q-1$ o-permutations, namely the non-zero scalar multiples of the o-polynomial, and  for  an o-permutation $f$ there is the unique o-polynomial $(1/f(1))f$.  For $f\in\cF$, $\<f\>$ will denote the one-dimensional subspace of $\cF$ containing $f$. By extending the previous notation, if $f$ is an o-permutation then we use $\cD(f)$ for  the oval $\{(1,t,f(t)): t \in\GF(q)\}\cup \{(0,1,0)\}$.
 
\begin{theorem}\cite[Theorem 4,Theorem 6]{okp}\label{th_1}
The magic action preserves the set of o-permutations. Furthermore, if $\cD(f)$ and $\cD(g)$, with $f,g$ o-permutations, are equivalent under $\PGammaL(3,q)$, then there is $\psi\in\PGammaL(2,q)$ such that $\psi f\in\<g\>$.
\end{theorem}
 
Let $\cH(\cC) = \{\cD(f_s):s \in \GF(q)\cup\{\infty\}\}$ and $\cH(\cC') = \{\cD(f_t'):t\in\GF(q)\cup\{\infty\}\}$ be herds. We say that $\cH(\cC)$ and $\cH(\cC')$ are {\em isomorphic} if there exists $\psi\in\PGammaL(2,q)$ such that for all $s\in\GF(q)\cup \{\infty\}$ we have $\psi f_s\in \langle f_t'\rangle$ under the magic action, and where the induced map $s \mapsto t$ is a permutation of $\GF(q) \cup \{\infty\}$. 

In the light of Theorem 16 in \cite{okp}, for $q$ even, we can re-write the Fundamental Theorem augmented with the following result.
\begin{theorem} The herds $\cH(\cC)$ and $\cH(\cC')$ are isomorphic if and only if $\mathrm{GQ}(\cC)$ and $\mathrm{GQ}(\cC')$ are isomorphic by an isomorphism mapping $(\infty)$ to $(\infty)$ and $(\underline 0, 0, \underline0)$ to $(\underline0, 0, \underline0)$.
\end{theorem} 

The above equivalences have been used by many researchers to classify the flocks of the quadratic cone for $13\le q\le 71$, $q\neq 64$.  
The following result summarizes the state of the art.

\begin{theorem} 
All flocks of the quadratic cone in $\PG(3,q)$ are classified for $q \le 71$, $q \ne 64$.
\end{theorem}

See \cite{thas1} for $q=2,3,4$; \cite{dcgt} for $q=5,7, 8$; \cite{my} for $q=9$; \cite{dch} for $q=11$; \cite{pr2} for $q=13$; \cite{dch,bokppr} for $q=16$; \cite{pr} for $q=17$; \cite{lp}  for $q=19, 23, 25, 27,29$; \cite{bokppr} for $q=32$; \cite{bet} for $q=31,37,41,43,47,49,53,59,61,67$; \cite{abc} for $q=71$.

In this paper we fill the gap by providing the classification of the flocks of the quadratic cone for $q=64$.

\section{Computational results} \label{sec_3}
Our aim in this section is to classify all flocks of the quadratic cone in $\PG(3,64)$, by bringing to bear our knowledge of ovals of $\PG(2,64)$, based on Vandendriessche's determination of the hyperovals of $\PG(2,64)$ \cite{vander}. It is clear that the most convenient model to approach this classification from this perspective is that of herds of ovals of $\PG(2,64)$. So, we aim to classify by computer all herds of ovals in $\PG(2,64)$, up to isomorphism. 

Let us now discuss the flocks of the quadratic cone in $\PG(3,64)$ arising from the flock quadrangle by means of Remark 2.2. For convenience we introduce the {\em Subiaco herd} as the herd of ovals associated with the Subiaco $q$-clan. Likewise, the {\em Adelaide herd} is the herd of ovals associated with the Adelaide $q$-clan.

For $q=64$ there are three known $q$-clans: the {\em classical} $q$-clan,   the {\em Subiaco}  $q$-clan  and the {\em Adelaide} $q$-clan, which we recall below. The $q$-clan associated with the linear flock for $q$ even is
\[
\cC_L: A_t= \begin{pmatrix}
t^{{1/2}} & t^{{1/2}} \\
0   & \kappa\, t^{{1/2}}
\end{pmatrix},
\] 
where $\Tr(\kappa)=1$.  The herd $\cH(\cC_L)$ consists of $q+1$ copies of a non-degenerate conic. The elation generalized quadrangle associated with this $q$-clan is isomorphic to $H(3,q^2)$ \cite{pt}. For this reason, both $\cC_L$ and $\cH(\cC_L)$ are called {\em classical}.

The {\em Subiaco}  $q$-clan, $q$ even,  \cite{cppr} is 
 \[
\cC_S: A_t= \begin{pmatrix}
f_0(t) & t^{{1/2}} \\
0   & \kappa\, f_\infty(t)
\end{pmatrix},
\] 

where, for some $\delta\in\GF(q)$, with $\delta^2 + \delta + 1\neq 0$ and $\Tr(1/\delta) = 1$, we have 

\[
\begin{aligned}
\kappa &=\frac{\delta^2 + \delta^5 + \delta^{1/2}}{\delta(1 +\delta + \delta^2)}\\[.1in]
f_0(t) &=\frac{\delta^2(t^4 +t)+\delta^2(1+\delta+\delta^2)(t^3 +t^2) }{(t^2 + \delta t + 1)^2}+t^{1/2}\\[.1in]
f_\infty(t) &=	\frac{\delta^4t^4 +\delta^3(1+\delta^2 +\delta^4)t^3 +\delta^3(1+\delta^2)t}{	(\delta^2 +\delta^5 +\delta^{1/2})(t^2 +\delta t +1)^2}	+ \frac{\delta^{1/2}}{(\delta^{2} +\delta^5 +\delta^{1/2})}t^{1/2}.
\end{aligned}
\]
\smallskip

For $q=2^e$, $e>2$ even,  $m\equiv \pm\frac{q-1}{3}(\mod\, q+1)$ and  some $\beta\in\GF(q^2)\setminus\{1\}$ with $\beta^{q+1}=1$, the {\em Adelaide} $q$-clan \cite{cop} is 
 \[
\cC_A: A_t= \begin{pmatrix}
f_0(t) & t^{{1/2}} \\
0   & \kappa\, f_\infty(t)
\end{pmatrix},
\] 

where 

\[
\begin{aligned}
\kappa &=\frac{T(\beta^m)}{T(\beta)}+\frac{1}{T(\beta^m)}+1\\[.1in]
f_0(t) &=\frac{T(\beta^m)(t+1)}{T(\beta)}+\frac{T((\beta t+\beta^q)^m)}{T(\beta)(t+T(\beta)t^{1/2}+1)^{m-1}}+t^{1/2}\\[.1in]
\kappa f_\infty(t) &=\frac{T(\beta^m)}{T(\beta)}t+\frac{T((\beta^2 t+1)^m)}{T(\beta)(T(\beta^m)(t+T(\beta)t^{1/2}+1)^{m-1}}+\frac{1}{T(\beta^m)}t^{1/2},
\end{aligned}
\]
with $T(x)=x+x^q$, for all $x\in\GF(q^2)$.

As the automorphism group of each corresponding flock quadrangle acts transitively on the lines through $(\infty)$, each of these quadrangles gives rise to only one class of equivalence of flocks of the quadratic cone.  The above property for the automorphism group  is true for the GQ arising from the classical $q$-clan since it is isomorphic to $H(3,q^2)$;  see \cite{blp,pay94,ppp} for the Subiaco $q$-clan and  \cite[Corollary 4.5]{cop} for the Adelaide $q$-clan, with $q=2^e$, $e\ge 6$. Therefore, up to isomorphism, there are at least three herds in $\PG(2,64)$ and so at least three flocks of the quadratic cone in $\PG(3,64)$. In addition, the ovals in the  Subiaco herd $\cH(\cC_S)$ belong to two different equivalence classes of ovals, one with stabilizer in $\PGammaL(3,64)$ of order 60, the other with stabilizer of order 15 \cite{ppp}; the Adelaide herd $\cH(\cC_A)$ consists of equivalent ovals whose stabilizer in $\PGammaL(3,64)$ has order 12 \cite{pr}.
 
\begin{theorem}
The flocks of the quadratic cone in $\PG(3,64)$ are, up to equivalence, the linear flock $\cJ(\cC_L)$, the Subiaco flock $\cJ(\cC_S)$ and the  Adelaide flock $\cJ(\cC_A)$.
\end{theorem}
\begin{proof}
In \cite{vander} it was shown that, up to equivalence, there are exactly four hyperovals in $\PG(2,64)$: the regular hyperoval, the Adelaide hyperoval, the Subiaco I hyperoval, the Subiaco II hyperoval. The corresponding o-polynomials can  be read off from the $q$-clans above: $t^{1/2}$ gives the regular hyperoval, $f_0$ from the Adelaide $q$-clan $ \cC_A$ gives the Adelaide hyperoval, $f_0$ from the Subiaco $q$-clan $\cC_S$ gives the Subiaco I hyperoval and $f_a$  from herd $\cH(\cC_S)$ arising from the Subiaco $q$-clan  gives the Subiaco II hyperoval, where $a^2+a+1=0$.

Since every oval lies in a unique hyperoval,  all the ovals in $\PG(2,64)$, up to equivalence, can be obtained by removing one ``representative''  from each point-orbit of the stabilizer in $\PGammaL(3,64)$ of each hyperoval, considered  as a permutation group acting on it. This operation produces 19 isomorphism classes of ovals in $\PG(2,64)$ (obtained by Siciliano on behalf of Penttila for a paper on the uniqueness of the inversive plane of order 64 \cite{pen}). By Theorem \ref{th_1}, the magic action produces all o-polynomials over $\GF(64)$, split in 19 {\em classes}. Then, taking into account the notion of isomorphism of herd, we intend to find all herds containing one of the above 19 ovals. 

According to the definition of herd, each oval is transformed to contain the points $(0, 1,0), (1,0,0), (1, 1, 1)$ and have nucleus $(0, 0, 1)$. Then, we obtain the associated o-polynomials by  Lagrange interpolation. Via the magic action, we get all the  ovals in $\PG(2,64)$ on the fundamental quadrangle, whose number is 17297346, and we store the corresponding o-polynomials just in terms of their coefficients.

Let $f_0$ and $f_{\infty}$ be two o-polynomials defining the herd $\cH=\{\cO_s:s\in\GF(q)\cup\{\infty\}\}$, where $\cO_s=\{(1, t, f_s(t)):t\in \GF(q)\}\cup\{(0,1,0)\}$, with $f_s(t)$ as in (\ref{eq_7}), for some $\kappa \in \GF(q)$, $\Tr(\kappa)=1$. By replacing $f_{\infty}$ with $f_s$, $s\neq 0$, the same herd arises but for $\kappa'=\kappa+s^{-1}+s^{-1/2}$. It is well known that every element of $\GF(q)$, $q$ even, of absolute trace 0 can be written as $\alpha+\alpha^2$, with $\alpha\in\GF(q)$ \cite{hir}. This implies that, as $s$ varies in $\GF(q)\setminus\{0\}$, $\kappa'$ describes all the elements of absolute trace 1 in $\GF(q)$. Thus, in order to detect all pairs $(f_0,f_{\infty})$ giving a herd, it suffices to fix an element of absolute trace 1, say $\kappa$. Since we are looking for herds in $\PG(2,64)$ up to isomorphism, we may fix $f_0$ to be one of the o-polynomials associated with the 19 ovals in $\PG(2,64)$. Because of the above arguments, for every such $f_0$, two distinct pairs of generators $(f_0,f_\infty)$ and $(f_0,\tilde f_\infty)$ will arise for each type of oval in each corresponding herd, except for the classical one for which all polynomials are equal to $t^{1/2}$.

At this point, once $f_0$ is fixed, for each o-polynomial $f_\infty$ in the 19 classes under the magic action, we check if $f_s$ is an o-permutation, first for $s=\kappa^{-2}$ (so that $1+\kappa s+s^{1/2}=1$). We find 25 such pairs. Having made this first selection, we proceed by checking the same condition, on all the above 25 pairs, for all non-zero $s$ in $\GF(64)$. Only 7 pairs survive, so only 7 herds are obtained. In order to recognize the type of each heard, we calculate the stabilizer in $\PGammaL(3,64)$ of all the ovals in the herd. Via this approach, we find one pair giving the classical herd, two pairs giving the Adelaide herd and four pairs giving the Subiaco herd. Therefore, every  herd in $\PG(2,64)$ is isomorphic to the classical herd, the Adelaide herd or the Subiaco herd, from which the complete classification of the flocks of the quadratic cone in $\PG(3,64)$  follows.
\end{proof}

\begin{remark}
{\em Our programs, written in MAGMA \cite{magma}, were run on a Mac OS system with one quad core Intel Core i7
 2.2 GHz processor. All the computational operations take approximately 56 hours of CPU time. The code is available on email request to the 3rd author. It is interesting to observe that writing a program to Lagrange interpolate over $\GF(64)$ is faster than the inbuilt Magma Interpolation command, and this difference in speed was necessary for us. In fact, we had to deal with 17297346 o-polynomials. In order to reduce considerably the memory used for these o-polynomials, we stored their coefficients as elements of the vector space $V(31,64)$, the coefficients of odd powers of $t$ being zero \cite{sb}. In this way we needed just 1.5 Gb of RAM versus 6 Gb if the same data were stored as a set of arrays.
}
 \end{remark}

\end{document}